\documentclass{daj}
\dajAUTHORdetails{%
  title = {A Sharp Threshold for van der Waerden's Theorem in Random Subsets},
  author = {Ehud Friedgut, Hiep Han, Yury Person, and Mathias Schacht},
  plaintextauthor = {Ehud Friedgut, H\d{\^e}p H\`an, Yury Person, Mathias Schacht},
  runningtitle = {Sharp Threshold for van der Waerden's Theorem in Random Subsets}, 
  keywords = {sharp thresholds, van der Waerden theorem},
}

\dajEDITORdetails{%
   year={2016},
   number={7},
   received={21 December 2015},   
   published={29 February 2016},  
   doi={10.19086/da.615},       
}   
     
\usepackage{amsmath,amssymb,amsthm} 
   
\usepackage{mathabx}\changenotsign    
\usepackage{mathrsfs} 
\usepackage{dsfont}
\usepackage{microtype}

\usepackage[open,openlevel=2,atend]{bookmark}

\usepackage[abbrev,msc-links,backrefs]{amsrefs} 
\usepackage{doi}

\renewcommand{\PrintDOI}[1]{\doi{#1}}

\usepackage[T1]{fontenc}
\usepackage{lmodern}

\usepackage[english]{babel}

\usepackage{enumitem}

\def\rmlabel{\upshape({\itshape \roman*\,})}

\def\alabel{\upshape({\itshape \alph*\,})}

\def\nlabel{\upshape({\itshape \arabic*\,})} 

\let\polishlcross=\l
\def\l{\ifmmode\ell\else\polishlcross\fi}

\def\tand{\ \text{and}\ }
\def\qand{\quad\text{and}\quad}
\def\qqand{\qquad\text{and}\qquad}

\let\emptyset=\varnothing
\let\setminus=\smallsetminus

\makeatletter
\def\moverlay{\mathpalette\mov@rlay}
\def\mov@rlay#1#2{\leavevmode\vtop{   \baselineskip\z@skip \lineskiplimit-\maxdimen
   \ialign{\hfil$\m@th#1##$\hfil\cr#2\crcr}}}
\newcommand{\charfusion}[3][\mathord]{
    #1{\ifx#1\mathop\vphantom{#2}\fi
        \mathpalette\mov@rlay{#2\cr#3}
      }
    \ifx#1\mathop\expandafter\displaylimits\fi}
\makeatother

\makeatletter
\newcommand{\subalign}[1]{  \vcenter{    \Let@ \restore@math@cr \default@tag
    \baselineskip\fontdimen10 \scriptfont\tw@
    \advance\baselineskip\fontdimen12 \scriptfont\tw@
    \lineskip\thr@@\fontdimen8 \scriptfont\thr@@
    \lineskiplimit\lineskip
    \ialign{\hfil$\m@th\scriptstyle##$&$\m@th\scriptstyle{}##$\crcr
      #1\crcr
    }  }
}
\makeatother  
  
\newtheorem{theorem}                   {Theorem} 
\newtheorem{lemma}           [theorem] {Lemma}

\newtheorem{cor}             [theorem] {Corollary}

\newtheorem{fact}            [theorem] {Fact}

\newtheorem{observation}      [theorem] {Observation}

\theoremstyle{remark}

\newcommand{\Z}{\mathds{Z}}
\newcommand{\Nat}{\mathds{N}}

\newcommand{\ex}[1]{\mathds{E}\left(#1\right)}
\newcommand{\pr}[1]{\mathds{P}\left(#1\right)}

\newcommand{\eps}{\varepsilon}
\newcommand{\cA}{\mathcal{A}}
\newcommand{\cC}{\mathcal{C}}

\newcommand{\cF}{\mathcal{F}}

\newcommand{\cT}{\mathcal{T}}
\newcommand{\cZ}{\mathcal{Z}}
\newcommand{\cB}{\mathcal{B}}

\newcommand{\rrr}{\rightarrow(k\text{-AP})_2}

\newcommand{\Zn}{\Z/n\Z}
\newcommand{\Znp}{\Z_{n,p}}
\newcommand{\Znep}{\Z_{n,\eps p}}

\renewcommand{\l}{\ell}
\newcommand{\cI}{\mathcal{I}}
\newcommand{\cS}{\mathcal{S}}
\newcommand{\ocF}{\overline{\cF}}
\newcommand{\cP}{\mathcal{P}}

\newcommand{\im}{\mathrm{Im}}

\newcommand{\EE}{\mathds{E}}
\renewcommand{\Pr}{\mathds{P}}
\newcommand{\codeg}{\mathrm{codeg}}

\def\phi{\varphi}

\newcommand{\NN}{\mathds{N}}
\newcommand{\RR}{\mathds{R}}

\begin{document}

\begin{frontmatter}[classification=text]

\title{A Sharp Threshold for van der Waerden's Theorem in Random Subsets\footnote{The 
	cooperation of the authors was supported by the \emph{German-Israeli Foundation 
	for Scientific Research and Development}.}}

\author[EF]{Ehud Friedgut\thanks{Supported by ISF grant 0398246.}}
\author[Hiep]{Hi{\d{\^e}}p H{\`a}n\thanks{Supported
	 by the Millennium Nucleus Information and Coordination in Networks ICM/FIC RC130003, 
	 by FONDECYT Iniciaci\'on grant 11150913, and by FAPESP (Proc. 2103/03447-6).}}
\author[YP]{Yury Person}
\author[MS]{Mathias Schacht\thanks{Supported through the \emph{Heisenberg-Programme} of the DFG.}}

\begin{abstract}
We establish sharpness for the threshold of van der Waerden's theorem in random subsets of $\Zn$.
More precisely, for $k\geq 3$ and $Z\subseteq \Zn$
we say~$Z$ has the van der Waerden property if any two-colouring of~$Z$ 
yields a monochromatic arithmetic progression of length~$k$. 
R\"odl and Ruci\'nski (1995) determined the threshold for this property for 
any~$k$ and we show that this threshold is sharp.

The proof is based on Friedgut's criterion~(1999) for sharp thresholds and on the recently developed 
\emph{container method} for independent sets in hypergraphs by 
Balogh, Morris and Samotij (2015) and by
Saxton and Thomason (2015).
\end{abstract}
\end{frontmatter}


\section{Introduction}
One of the main research directions in extremal and probabilistic combinatorics over the last two decades has been the extension of classical 
results for discrete structures to the sparse random setting. 
Prime examples include Ramsey's theorem for graphs and hypergraphs~\cites{RR95,FRS10,CG11}, 
Tur\'an's theorem in extremal graph theory, and Szemer\'edi's theorem on arithmetic progressions~\cites{mathias,CG11} 
(see also~\cites{BMS12,ST12,CGSS14}). Results of that form establish the \emph{threshold} for the 
classical result in the random setting. For a property the threshold is given by a function~$\hat p=\hat p(n)$ 
such that for every $p_0\ll \hat p$ the random graph $G(n,p_0)$ (or a random binomial subset of~$[n]=\{1,2,\ldots, n\}$)
with parameter~$p_0$, the probability the property
holds is asymptotically zero,
whereas if $p_0$ is replaced by 
some $p_1\gg \hat p$ the property does hold asymptotically almost surely (a.a.s.), i.e.,
for a property~$\cP$ of graphs and probabilities~$p=p(n)$ we have 
\[
	\lim_{n\to\infty}\Pr\big(G(n,p)\in\cP\big)
	=
	\begin{cases}
	0, &\text{if}\ p\ll\hat p\\
	1,&\text{if}\ p\gg\hat p\,.
	\end{cases}
\]
The two statements involving $p_0$ and $p_1$
are referred to as the \emph{0-statement} and the \emph{1-statement}. For the properties mentioned above it can be shown that the optimal parameters~$p_0$ and $p_1$ for which
the $0$-statement and the $1$-statement hold, only differ by a
multiplicative constant. 

The threshold for van der Waerden's theorem~\cite{vdW} is such an example and was obtained by 
R\"odl and Ruci\'nski in~\cites{RR95,RR97}. We denote by $[n]_p$ the binomial random subset of $[n]$, where every element of 
$[n]$ is included independently with probability $p=p(n)$. Furthermore, for a subset~$A\subseteq [n]$ 
we write  $A\rightarrow(k\text{-AP})_r$ to denote 
the fact that no matter how one colours the elements of~$A$ with $r$ colours there is 
always a monochromatic arithmetic progression with $k$ elements in~$A$.

\begin{theorem}[R\"odl \& Ruci\'nski]
\label{thm:randomvdW}
For every $k\geq 3$ and $r\geq 2$ there exist constants $c_0$, $c_1>0$ such that 
\[
\lim_{n\to\infty}\Pr([n]_p\rightarrow (k\text{-AP})_r)=
\begin{cases}
0, &\text{ if }p\le c_0n^{-\frac{1}{k-1}},\\
1, &\text{ if }p\ge c_1n^{-\frac{1}{k-1}}.
\end{cases}
\]
\end{theorem}

For the corresponding result in $\Zn$ and for two colours
we close the gap between~$c_0$ and~$c_1$. More precisely,  
we show that there exist bounded sequences $c_0(n)$ and $c_1(n)$ with ratio 
tending to 1 as $n$ tends to infinity such that the statement holds
(see Theorem~\ref{thm:main} below). 
In other words, we establish a {\em sharp threshold} for 
van der Waerden's theorem for two colours in $\Zn$.

Similarly to the situation for subsets of $[n]$
we write $A\rightarrow(k\text{-AP})_r$ for subsets $A\subseteq \Z/n\Z$ 
if any $r$-colouring of $A$ yields a monochromatic  arithmetic progression with $k$-elements in~$\Zn$
and we write $A\nrightarrow(k\text{-AP})_r$ if $A$ fails to have this property.
Moreover, we denote by $\Znp$
the binomial random subset of $\Zn$ with parameter $p$. With this notation at hand we can state our main result.

\begin{theorem}\label{thm:main}
For all $k\geq 3$ there exist constants $c_1>c_0>0$ and a function $c(n)$ satisfying $c_0\leq c(n)\leq c_1$ 
such that for every $\eps>0$ we have
\[
\lim_{n\to\infty}\Pr(\Znp\rightarrow (k\text{-AP})_2)=
\begin{cases}
0, &\text{ if }p\le (1-\eps)c(n)n^{-\frac{1}{k-1}},\\
1, &\text{ if }p\ge (1+\eps)c(n)n^{-\frac{1}{k-1}}.
\end{cases}
\]
\end{theorem}
We have to insist on the setting of $\Z/n\Z$ (instead of $[n]$) since the 
symmetry will play a small but crucial r\^ole in our proof. Another shortcoming is the 
restriction to two colours~$r=2$ and we believe it would be very interesting to extend the result 
to arbitrary~$r$. We remark that only a few sharp thresholds for Ramsey properties 
are known (see, e.g.,~\cites{FK00,FrRoRuTe06}) so far.

Among other tools our proof relies heavily on the criterion for sharp thresholds of 
Friedgut and its extension due to Bourgain~\cite{Fri99}. 
Another crucial tool is the recent \emph{container theorem} for independent sets in hypergraphs 
due to Balogh, Morris and Samotij~\cite{BMS12} and Thomason and Saxton~\cite{ST12}.

Our proof extends to other Ramsey properties for two colours, as long as the corresponding extremal 
problem is \emph{degenerate}, i.e., positive density yields many copies of the target structure 
and the target structure is strictly balanced with respect to its so-called \emph{$2$-density}.
For example, even cycles in graphs, complete $k$-partite, $k$-uniform hypergraphs, and strictly balanced, density regular Rado systems (see~\cite{RR97})
satisfy these assumptions. Moreover, Schacht and Schulenburg~\cite{SchSch} noted that 
the approach undertaken here can be refined to give a shorter proof for the sharp threshold 
of the Ramsey property for triangles and two colours from~\cite{FrRoRuTe06} 
and, more generally, for arbitrary odd cycles.

\section{Locality of coarse thresholds}
In \cite{Fri99} Friedgut gave a necessary condition for a graph property to have a coarse threshold, namely, that it is approximable by a ``local'' property. In the appendix to this work 
Bourgain proved a similar result for more general discrete structures. Here we 
state the special case applicable for properties in $\Zn$.
\begin{theorem}[Bourgain]\label{thm:Bourgain}
There exist functions $\delta(C,\tau)$ and $K(C,\tau)$ such that the following holds. Let $p=o(1)$ 
as  $n$ tends to infinity, let $\cA$ be a monotone family of subsets of $\Zn$, with
\[
\tau<\mu(p,\cA):=\Pr(\Znp\in\cA)<1-\tau\,,
\] and assume also $p\cdot \frac{d\mu(p,\cA)}{dp}\le C$.
 Then there exists some $B\subseteq \Zn$ with $|B|\le K$ such that 
\begin{equation}\label{eq:bconc}
   \Pr(\Znp\in\cA\,|\, B\subseteq \Znp)>\Pr(\Znp\in\cA)+\delta\,.
\end{equation} 
\end{theorem}
  
Note that whenever a property $\cA$ (or rather, a series of properties $\cA_n$) has a coarse threshold there exist constants $C$ and $\tau$ such that for infinitely many values of $n$ the hypothesis of the theorem holds.
For applications, it would be problematic if there exists a $B$ with
$|B| \le K$ and~$B \in \cA$, since this would trivialise the conclusion~\eqref{eq:bconc}.  However, as observed in~\cite{Fri05}, the above theorem can be strengthened, without modifying the original proof, to deduce that the set of $B$'s for which the assertion holds has non-negligible measure, i.e.,
there exists a family~$\mathcal{B}$ such that 
\[
 \Pr(B \subseteq \Znp\ \text{for some}\ B\in\cB) > \eta,
\]
where $\eta> 0 $ depends only on $C$ and $\tau$ but not on $n$,
and  every $B \in \mathcal{B}$ satisfies the conclusion of Theorem~\ref{thm:Bourgain}, i.e.,  $|B| \le K$ and
  \[
   \Pr(\Znp\in\cA\,|\, B\subseteq \Znp)>\Pr(\Znp\in\cA)+\delta\,.
  \]

This allows us to make assumptions about $B$ in the application below, as long as the set of~$B$'s violating the assumptions has negligible measure. 
In particular, Lemma~\ref{lem:Vojta} below implies that any collection
of sets~$B\subseteq\Zn$, each of bounded size and 
with $B \rightarrow (k\text{-AP})_2$, appear only with probability tending to 
zero in $\Znp$ for $p=O(n^{-\frac{1}{k-1}})$. Consequently, in our proof we can therefore assume 
that the set $B$, provided by Theorem~\ref{thm:Bourgain} on the assumption that 
$\Znp\rightarrow(k\text{-AP})_2$ has a coarse thresholds, itself fails to have the van der Waerden property, i.e.,
$B\nrightarrow(k\text{-AP})_2$.
\begin{lemma}
\label{lem:Vojta}
 Let $\cB$ be a family of subsets of $\Zn$ with the property that every $B\in\cB$ satisfies 
 $|B|\leq \log\log n$ and $B\rrr$. Then for every $c>0$ and every
 sequence of probabilities $p=p(n)\leq cn^{-\frac{1}{k-1}}$ we have $\Pr(B\subseteq\Znp\ \text{for some}\ B\in\cB)=o(1)$. 
\end{lemma}
Lemma~\ref{lem:Vojta} was implicitly proved in~\cite{RR97}*{Section~7}
(see the Deterministic and the Probabilistic Lemma there). For completeness 
we include a sketch of the proof here.
\begin{proof}[Proof (Sketch)] Let $k\geq 3$ be an integer and let $p=p(n)\leq cn^{-1/(k-1)}$ 
	for some $c\in\RR_{>0}$. For a set $Z\subseteq \Zn$ we consider the auxiliary 
	$k$-uniform hypergraph $H_{Z,k}$ with vertex set~$Z$ and hyperedges 
	corresponding to $k$-APs in $Z$. The Deterministic Lemma~\cite{RR97}*{p.~500}
	asserts that if~$Z\rightarrow(k\text{-AP})_2$, then one of the following configurations 
	must appear as a subhypergraph of~$H_{Z,k}$:
	\begin{enumerate}[label=\rmlabel]
	\item either $H_{Z,k}$ contains a subhypergraph of type $\cT_1$
	consisting of a hyperedge $e_0$ and a
	loose cycle $C_\l$
	of some length~$\l\geq 3$, i.e., $C_\l$ consists of $\l$ hyperedges 
	$e_1,\dots,e_{\l}$ satisfying for 
	every $1\leq i<j\leq \l$
	\[
		|e_i\cap e_j|
		=
		\begin{cases}
			1\,,&\text{if}\ j=i+1\ \text{or}\ (i=1\tand j=\l)\\
			0\,,&\text{otherwise}\,,
		\end{cases}
	\]
	while the additional hyperedge $e_0$ has at least one vertex outside
	the cycle and shares at least two vertices with the cycles, i.e.,
	\[
		2\leq |e_0\cap V(C_\l)|<k\,;
	\]
	\item or $H_{Z,k}$ contains a subhypergraph of type $\cT_2$
	consisting of a non-induced loose path~$P_\l$
	of some length $\l\geq 2$, i.e.,  $\l$ hyperedges $e_1,\dots,e_{\l}$
	satisfying for 
	$1\leq i<j\leq \l$
	\[
		|e_i\cap e_j|
		=
		\begin{cases}
			1\,,&\text{if}\ j=i+1\\
			0\,,&\text{otherwise}\,.
		\end{cases}
	\]
	The condition that $P_\l$ is non-induced means 
	that there exists some hyperedge $e_0$ in $E(H_{Z,k})\setminus E(P_\l)$
	such that $e_0\subseteq V(P_\l)$.
	\end{enumerate}
	
	A simple first moment argument shows that a.a.s.\ 
	no hypergraph on at most $\log n$ vertices 
	of types $\cT_1$ or $\cT_2$ appears in $H_{\Znp,k}$.
	For that let $X_1$ (resp.\ $X_2$) 
	be the random variable counting the number of 
	copies of hypergraphs of type $\cT_1$ (resp.\ $\cT_2$) on at most~$\log n$ vertices in $\Znp$. Below we show that there exists a constant $K=K(k,c)$ such that 
	\begin{equation}\label{eq:X1X2}
		\EE X_1 \leq (\log n)^K\cdot p
		\qqand 
		\EE X_2 \leq (\log n)^K\cdot p
	\end{equation}
	and since $p\leq cn^{-1/(k-1)}\ll (\log n)^{-K}$
	the lemma follows from Markov's inequality.
	
	We start with the random variable~$X_1$. Since the hyperedges of $H_{\Znp,k}$
	correspond to $k$-APs the number $Y_\l$ of  loose cycles $C_\l$ of length at 
	least $\l\geq 3$ in 
	$H_{\Znp,k}$ satisfies
	\[
		\EE Y_\l \leq O(p^{(k-1)\l}\cdot n^\l)=O(c^{(k-1)\l}n^{-\l}\cdot n^\l)=O(c^{k\l})\,.
	\]
	For a given loose cycle $C_\l$ the additional hyperedge $e_0$ (to complete $C_\l$ to 
	a hypergraph of type~$\cT_1$) shares at least two vertices with $C_\l$. However, 
	with these two vertices fixed there are less than $k^2$ possibilities 
	to complete this choice to a $k$-AP in $\Zn$, i.e., these two vertices can be completed 
	in at most $k^2$ ways to form the hyperedge~$e_0$ in $H_{\Zn,k}$.
	In other words, for a fixed loose cycle $C_\l$ on $(k-1)\l$ vertices there 
	are at most~\mbox{$(k-1)^2\l^2\cdot k^2$} possibilities to complete~$C_\l$ to a hypergraph 
	of type~$\cT_1$. Furthermore, since the hyperedge~$e_0$ is required to have 
	at least one vertex outside~$C_\l$ we have
	\[
		\EE X_1 
		< 
		\sum_{\l=3}^{\log\log n}k^4\l^2 p\cdot \EE Y_\l 
		=
		\sum_{\l=3}^{\log\log n}O(\l^2 p\cdot c^{k\l})
		\leq 
		(\log n)^K\cdot p
	\]
	for some constant $K=K(k,c)$,
	which establishes the first estimate in~\eqref{eq:X1X2}.
	
	Similarly, for the random variable $X_2$ we first observe that 
	the expected number $Y'_\l$ of  loose paths $P_\l$ of length at 
	least $\l\geq 2$ in $\Znp$ satisfies
	\[
		\EE Y'_\l \leq O(p^kn^2\cdot p^{(k-1)(\l-1)} n^{(\l-1)})=O(pn\cdot c^{k\l})\,.
	\]
	Since the additional hyperedge~$e_0$ reduces the expected number of choices for 
	at least one of the hyperedges of $P_\l$ from $O(p^{k-1}n)$ to 
	$O(p^{k-1})$ and since $e_0$ is fixed after selecting two of its vertices within $P_\l$ 
	we arrive at 
	\[
		\EE X_2 \leq 
		\sum_{\l=2}^{\log\log n}O\left(\frac{\l^2}{n}\right)\cdot \EE Y'_\l 
		=
		\sum_{\l=2}^{\log\log n}O(\l^2p\cdot c^{k\l})
		\leq
		(\log n)^K\cdot p
	\]
	for some constant $K=K(k,c)$,
	which establishes the second estimate in~\eqref{eq:X1X2}
	and concludes the proof of the lemma.
\end{proof}

We summarise the discussion above in the following corollary of Theorem~\ref{thm:Bourgain}, which is tailored for our 
proof of Theorem~\ref{thm:main}.

\begin{cor}\label{thm:BF}
Assume that the property $\{Z\subseteq \Zn\colon Z\rrr\}$ does not
have a sharp threshold. 
Then there exist constants $c_1$, $c_0$, $\alpha$, $\eps$, $\mu >0 $, and $K$ and a function 
$c(n)\colon \NN\to \RR$ with $c_0<c(n)<c_1$ 
so that for infinitely many values of $n$ and $p=c(n)n^{-\frac{1}{k-1}}$ the following holds.
 
There exists a subset $B$ of $\Zn$ of size at most $K$ with 
$B\not\rightarrow (k\text{-AP})_2$ such that for every family $\cZ$ of subsets from $\Zn$ 
satisfying $\Pr(\Znp\in\cZ) > 1-\mu$ 
there exists a $Z \in \cZ$ so that 
\begin{enumerate}[label=\alabel]
\item \label{item:BF1} $\Pr(Z \cup (B+x)\rightarrow(k\text{-AP})_2) > \alpha$, where $x\in \Zn$ is chosen uniformly at random, and 
\item\label{item:BF2} $\Pr(Z\cup \Z_{n,\eps p}\rightarrow(k\text{-AP})_2) < \alpha/2 $.
\end{enumerate}
\end{cor}
We remark that the $\Pr(\cdot)$ in Corollary~\ref{thm:BF} concern different probability spaces. 
While the assumption $\Pr(\Znp\in\cZ) > 1-\mu$  concerns the binomial random subset $\Znp$,
we consider~$x$ chosen uniformly at random from $\Zn$ in~\ref{item:BF1} and the binomial random subset $\Z_{n,\eps p}$
in~\ref{item:BF2}. 
We close this section with a short sketch of the proof of Corollary~\ref{thm:BF}.

\begin{proof}[Proof (Sketch).]
For $k\geq 3$ we consider the property $\cA=\{Z\subseteq \Zn\colon Z\rrr\}$ and assume that it does not have 
a sharp threshold. Consequently, there exists a function $p=p(n)$ such that for infinitely many~$n$ 
the assumptions of Theorem~\ref{thm:Bourgain} hold, which implicitly yields constants $C$, $\tau$, $\delta$, and $K$.
Let $\bar\cA=\cP(\Zn)\setminus \cA$ be the family of subsets of $\Zn$ that fail to have the van der Waerden property.
Since we assume that the threshold for $\cA$ is not sharp, we may fix $\eps>0$ sufficiently small, such that there must 
be some $\alpha$ with  $ \delta/2 > \alpha > 0$ so that if we let $\cZ'\subseteq \bar\cA$ be the sets $Z\in\bar\cA$ 
for which  
\begin{equation}\label{simple}
\Pr(Z\cup\Z_{n,\eps p} \rightarrow(k\text{-AP})_2) < \alpha/2 
\end{equation}
then $\Pr(\Znp\in\cZ'\,|\,\Znp\in\bar\cA)\geq 1-\delta/4$.

Also for  $p=p(n)$ we have $\tau < \Pr(\Znp\rrr) = \Pr(\Znp\in\cA) < 1-\tau$, so by Theorem~\ref{thm:randomvdW} there exist some constants $c_1\geq c_0>0$
such that 
$p=p(n)=c(n)n^{-\frac{1}{k-1}}$ for some function $c(n)\colon \NN\to \RR$ satisfying $c_0\leq c(n)\leq c_1$.
Strictly speaking, we should use the version of Theorem~\ref{thm:randomvdW} for $\Zn$ instead of $[n]$. However, 
it is easy to see that the $1$-statement for random subsets of~$[n]$ implies the $1$-statement for 
random subsets of $\Zn$ (up to a different constant $c_1$) and
the proof of the $0$-statement from~\cite{RR97} can be straightforwardly adjusted for subsets of $\Zn$.

Moreover, for any such $n$ Theorem~\ref{thm:Bourgain} yields a family $\cB$ of subsets of $\Zn$ each of size at most $K$
such that~\eqref{eq:bconc} holds and an element of $\cB$ appears as a subset of $\Znp$ with probability at least $\eta$.
Consequently, Lemma~\ref{lem:Vojta} asserts that at least one such $B\in \cB$ fails to have the van der Waerden property itself, i.e., 
$B\not\rightarrow (k\text{-AP})_2$. By symmetry it follows from~\eqref{eq:bconc}, that the same holds for every translate $B+x$
with $x\in \Zn$. In particular, consider the family $\cZ''\subseteq \bar\cA$ of all sets $Z\in\bar\cA$ such that for at least $(\delta/2) n$ translates $B+x$
we have $Z\cup (B+x)\rrr$, i.e., 
\[
	\Pr(Z \cup (B+x)\rightarrow(k\text{-AP})_2) > \delta/2 > \alpha
\]
for $x$ chosen uniformly at random from $\Zn$. Then $\Pr(\Znp\in\cZ''\,|\,\Znp\in\bar\cA)\geq \delta/2.$
  So, taking  $\mu < \delta\cdot\Pr(\Znp\in\bar\cA)/8$ we  have that if   $\Pr(\Znp\in\cZ)\geq 1-\mu$ 
then $\cZ \cap \cZ' \cap \cZ'' \neq \emptyset$. Any $Z$ in this non-empty family has the desired properties.
\end{proof}

\section{Lemmas and the proof of the main theorem}
In this section we state all the necessary notation and lemmas to give the proof of Theorem~\ref{thm:main}. 
We start with an outline of this proof.
\subsection{Outline of the proof}
The point of departure is Corollary~\ref{thm:BF}  and we will derive a contradiction to its second property.
To this end, we consider an appropriate set $Z$ as given by Corollary~\ref{thm:BF} and let 
$\Phi$ denote the set of all colourings of $Z$ without a monochromatic $k$-AP. The main obstacle 
is to find a partition of~$\Phi$ into $i_0=2^{o(pn)}$ 
classes $\Phi_1,\dots,\Phi_{i_0}$, such that  any two colourings $\phi$, $\phi'$ from any partition class~$\Phi_i$ agree on a relatively dense subset~$C_i$ of~$Z$, i.e.\
$\phi(z)=\phi'(z)$ for all $z\in C_i$. 
Let~$B_i$ denote the larger monochromatic subset of~$C_i$, say  of colour blue.
We consider the the set $F(B_i)$ of those elements in $\Zn$, which extend 
a blue $(k-1)$-AP in $B_i$ to an $k$-AP.
Note that Corollary~\ref{thm:BF} allows us to impose
further conditions on~$Z$ as long as $\Znp$ satisfies them almost surely. 
One of these properties will assert that $F(B_i)$ is of size linear in~$n$
Consequently, by a quantitative version of 
Szemer\'edi's theorem we know that the number of $k$-APs in the focus of $B_i$ 
is~$\Omega(n^2)$. 
Consider $U_i=(\Zn)_{\eps p}\cap F(B_i)$ and note that if any element of $U_i$ is coloured blue then this induces a blue $k$-AP with $Z$ under any colouring
$\phi\in\Phi_i$. 
Hence, to extend any $\phi\in\Phi_i$ to a colouring 
of $Z\cup U_i$ without a monochromatic $k$-AP it is necessary that all elements in $U_i$ are coloured red. Consequently, the probability of
a successful extension of any colouring in $\Phi_i$ is bounded from above by the probability that $U_i$ does not contain a $k$-AP. This, however, is at most 
$\exp(-\Omega(p^kn^2))=\exp(-\Omega(pn))$ by Janson's inequality. 
We conclude by the union bound that after the second round, i.e.~$\Znep$, the probability 
that any $k$-AP-free colouring of $Z$ survives is $i_0\exp(-\Omega(pn))=o(1)$ which contradicts~\ref{item:BF2}.

To establish the above mentioned partition of $\Phi$  
we will define an auxiliary hypergraph~$H$ in such a way that every $\phi\in \Phi$ can be associated 
with a hitting set of $H$. As the complements of hitting sets are independent and as $H$ will be ``well-behaved''
we can apply a structural result of Balogh, Morris and Samotij~\cite{BMS12} on independent sets in uniform hypergraphs 
(see Theorem~\ref{thm:BMS}) to ``capture'' the hitting sets 
of $H$ and hence a partition of $\Phi$ with the properties mentioned
above (see Lemma~\ref{lem:cores}). We remark, that the proof of showing that 
$H$ is indeed well-behaved will use~\ref{item:BF1} and additional properties 
of $Z$, which hold a.a.s.\ in $\Znp$.
Next we will introduce the necessary concepts along with the lemmas needed to give the proof of the main theorem.
The proof of Theorem~\ref{thm:main} will then be given in Section~\ref{sec:proofmain}.

\subsection{Lemmas}
We  call $(Z,B)$ an \emph{interacting pair} if
$Z\not\rightarrow(k\text{-AP})_2$ and $B\not\rightarrow(k\text{-AP})_2$ but $Z\cup B\rightarrow(k\text{-AP})_2$.
Further,  $(Z,B,X)$ is called an \emph{interacting triple} if  $(Z,B+x)$ is interacting for all $x\in X$.
Note that Corollary~\ref{thm:BF} asserts  that there is an interacting triple $(Z,B,X)$ with  $|X|>\alpha n$.
In the following we shall concentrate on elements  which are decisive for interactions. 
Given a (not necessarily interacting) pair $(A,B)$
we say that an element $a\in A$ \emph{focuses} on $b\in B+x$ if there are $k-2$ further elements in $A\cup B$  forming a $k$-AP with $a$ and~$b$.

The set of vertices of particular interest, given a pair $(A,B)$,  is 
\[M(A,B)=\{a\in A\colon \text{there is a } b\in B \text{ such that $a$ focuses on } b\}\]
and for a triple $(A,B,X)$ we define the hypergraph $H=H(A,B,X)$ with the vertex set~$A$ and 
the edge set consisting of all $M(A,B+x)$ with $x\in X$. 
We are interested in the hypergraph $H(Z,B,X)$ with an interacting triple  $(Z,B,X)$.
We will make use of the fact that Corollary~\ref{thm:BF} allows us to put further restrictions on $Z$ as long these events occur a.a.s.\ for $\Znp$.
The requirement we want to make is that  the maximum degree and co-degree of $H(Z,B,X)$ are well behaved.

\begin{lemma}
\label{lem:degreeH}
For given $c_1, k, K$ and all $B\subset \Z/n\Z$ of size $|B|\leq K$ the following holds a.a.s.\ for $2\log n\leq p\leq c_1n^{-1/(k-1)}$:
There is a set $Y\subset \Z/n\Z$ of size at most $n^{1-1/(k-1)}\log n$ such that the hypergraph $H=H(\Znp,B, (\Z/n\Z)\setminus Y)$ satisfies
\begin{enumerate}[label=\nlabel]
\item $2pn\geq v(H)\geq pn/2,$
\item $\Delta_1(H)\leq 10k^3Kp^{k-2}n$, and
\item $\Delta_2(H)\leq 8\log n$,
\end{enumerate}
where $\Delta_1(H)$ and $\Delta_2(H)$ denote the maximum vertex degree of~$H$
and maximum co-degree of pairs of vertices of~$H$.

\end{lemma}
\noindent
We postpone the proof of Lemma~\ref{lem:degreeH}. It can be found in Section~\ref{sec:lem6and10}.

A set of vertices of a hypergraph is called a hitting set if it intersects every edge of this hypergraph.
The conditions in Lemma~\ref{lem:degreeH} will be used to control  
 the hitting sets of $H(Z,B,X)$ which play an important r\^ole as explained in the following. A colouring of a set is called $k$-AP free if it 
 does not exhibit a monochromatic $k$-AP. For an interacting triple~$(Z,B,X)$ we
fix a $k$-AP free colouring $\sigma\colon B\to\{\text{red,blue}\}$ of $B$, which 
exists since $B\not\rightarrow(k\text{-AP})_2$.
We also consider the ``same'' colouring for all its translates $B+x$. 
More precisely, let $B=\{b_1,\dots,b_{|B|}\}$ and
for every $x\in X$ we consider 
the $k$-AP free colouring $\sigma_x\colon(B+x)\to \{\text{red,blue}\}$ 
defined by 
$\sigma_x(b_i+x)=\sigma(b_i)$.

For  any $k$-AP-free colouring $\phi$ of $Z$ and any $x\in X$ the colouring of $Z\cup (B+x)$ induced by~$\sigma_x$ and $\phi$ 
must exhibit a monochromatic $k$-AP (intersecting both $Z$ and $B+x$)  since $(Z,B+x)$ is interacting. Hence,
for each $x\in X$ the edge $M(Z,B+x)$ contains
an element~$z$ focussing on an element
$b\in B+x$ such that  $\phi(z)=\sigma_x(b)$.
Such a vertex $z\in M(Z,B+x)$ we call \emph{activated} by $\sigma_x$ and $\phi$ and we define the set of activated 
vertices
\[A^{\sigma_x}_\phi(Z,B+x)=\{z\in Z\colon z\text{ is activated by $\sigma_x$ and }\phi\}\]
which is a non-empty subset of $M(Z,B+x)$. 
\begin{observation}
\label{obs:hittingset}
Suppose that we are given an interacting triple $(Z,B,X)$, 
a $k$-AP free colouring 
$\sigma\colon B\to \{\text{red,blue}\}$ of $B=\{b_1,\dots,b_{|B|}\}$, and suppose for every $x\in X$
the translate $B+x$ is coloured with the same pattern $\sigma_x$. 
Further, let $\phi$ be a $k$-AP free colouring of $Z$. Then the set of activated vertices
\[A_\phi=A^\sigma_\phi(Z,B,X)=\bigcup_{x\in X}A^{\sigma_x}_\phi(Z,B+x)\]
is a hitting set of $H(Z,B,X)$. 
\end{observation}

The following lemma shows that the hitting sets of  well-behaved uniform hypergraphs 
can be ``captured'' by a small number of  sets of large size 
called cores. 

\vbox{
\begin{lemma}\label{lem:cores}
For every natural $k\geq 3$, $\ell\ge 2$ and all positive $C_0$, $C_1$ there are $C'$ and
$\beta>0$ such that the following holds. 

If $H$ is an $\ell$-uniform hypergraph with $m$ vertices, $C_0 m^{1+1/(k-2)}$ edges,
$\Delta_1(H)\le C_1 m^{1/(k-2)}$, and $\Delta_2(H)\le C_1 \log m$ then there is a family $\cC$ of subsets of $V(H)$, which we shall call \emph{cores}, such that
\begin{enumerate}[label=\rmlabel]
\item\label{it:cores1} for $t=1-1/(k-2)(\ell-1)$ we have \[|\cC|\leq \sum_{i=1}^{C'm^t}\binom{m}{i}\]
\item\label{it:cores2} $|C|\ge \beta m$ for every $C\in \cC$, and
\item\label{it:cores3} every hitting set of $H$ contains some $C$ from $\cC$.
\end{enumerate}
\end{lemma}
}
Lemma~\ref{lem:cores} will follow from the main result from~\cite{BMS12}. The proof can be found in Section~\ref{sec:cores}.
As it turns out, we can insist that  the interacting triple $(Z,B,X)$ guaranteed by Corollary~\ref{thm:BF} has the additional property that $X$ contains a suitable subset  
$X'\subset X$ so that the hypergraph $H(Z,B,X')$ is uniform. In this case
Lemma~\ref{lem:cores} allows us to partition the sets of all $k$-AP free colourings of $Z$ into a small number of partition classes 
$\{\Phi_C\}_{C\in \cC}$, each  represented by a big core~$C\in \cC$.

However, we wish to refine the partition classes further so that every two colourings~$\phi$,~$\phi'\in\Phi_C$  from the same partition class agree on a large vertex set. 
This can be accomplished by applying Lemma~\ref{lem:cores} to $H(Z,B,X')$ for a more refined   subset $X'\subset X$. 
Indeed, we will make sure that there is a set which guarantees 
that the colours of the activated vertices $A_{\phi}$ under $\phi$ as defined in Observation~\ref{obs:hittingset}  are already ``determined'' 
by~$\sigma$. 
This implies that any two colourings $\phi$, $\phi'\in\Phi_C$ agree on $A_{\phi}\cap A_{\phi'}$, hence, on the core~$C$ representing them, i.e.\
$\phi(z)=\phi'(z)$ for all $z\in C$.
To make this formal we need the following definitions.

In the following we fix some linear order
on the elements of $\Zn$, which we denote simply by $<$.
A triple $(Z,B,X)$ is called \emph{regular} if for all $x\in X$
every element of~$Z$ focuses on at most one element in $B+x$. 
Given   a regular triple $(Z,B,X)$ and an $x\in X$ 
let $z_1<\dots< z_\ell$ denote the  elements of $M_x=M(Z,B+x)\in H(Z,B,X)$. 
We say that $z\in M(Z,B+x)$ has \emph{index} $i$ if $z=z_i$ and the triple $(Z,B,X)$ is called \emph{index consistent}
if   for any element $z\in Z$ and any two 
edges~$M_x$,~$M_{x'}$ containing $z$ the indices of $z$ in $M_x$ and~$M_{x'}$ are the same.

Further, let $B=\{b_1,\dots,b_{|B|}\}$. We associate to the edge $M_x$ its \emph{profile} which is the function  
$\pi\colon  [\ell]\to [|B|]$ indicating which $z_i$ focusses on which~$b_j+x$, formally: $\pi(i)=j$ if $z_i$ 
focusses on~$b_j+x$. Since $(Z,B,X)$ is regular, 
each $z\in M_x$ focuses on exactly one element from $B+x$, thus, the profile of $M_x$ is well-defined and unique. We call $\ell$ the \emph{length} of the profile and we say 
that the triple $(Z,B,X)$ has profile $\pi$ with length $\ell$ if all edges of $H(Z,B,X)$ do.
We summarise the desired properties for the hitting sets of $H(Z,B,X)$ associated to  $k$-AP free colourings  of $Z$. 
\vbox{\begin{observation}
\label{obs:hittingset2}
Fix some linear order on~$\Zn$.
Suppose the triple $(Z,B,X)$ in Observation~\ref{obs:hittingset} 
with $B=\{b_1,\dots,b_{|B|}\}$ is index consistent and has profile $\pi$.  
Let  $A_\phi=A^\sigma_\phi(Z,B,X)$ be the vertex set activated by $\phi$ and $\sigma$ as defined in Observation~\ref{obs:hittingset}.
Then for any vertex $z\in A_\phi$ the colour~$\phi(z)$ of $z$ is already determined 
by~$\sigma$ and the (unique) index  $i$ of $z$, indeed, 
$\phi(z)=\sigma(b_{\pi(i)})$. In particular, any two  $k$-AP free colourings  $\phi$ and $\phi'$ of $Z$ agree on $A_{\phi}\cap A_{\phi'}$, i.e.\ $\phi(z)=\phi'(z)$ for all 
$z\in A_{\phi}\cap A_{\phi'}$.
\end{observation}}

The following lemma will allow us to restrict considerations to  index consistent triples  with a bounded length profile.

\begin{lemma}
\label{lem:goodtriple} For all $c_1>0$, $k$, $K$ and $\alpha>0$ there exist $L$ and $\alpha'>0$ such that 
for all $B\subset \Z/n\Z$ of size $|B|\leq K$ and $p\leq c_1n^{-1/(k-1)}$ the following holds a.a.s.

For any linear order on~$\Zn$ there is a set $Y_n\subset \Z/n\Z$ of size at most $n^{1-1/(k-1)}\log n$ such that 
 for every set $X\subset \Z/n\Z$ of size $X\geq \alpha n$ there is a set $X'\subset X\setminus Y_n$ of size 
 $|X'|\geq \alpha' n$ and a profile $\pi$
of length at most $L$ such that $(\Znp,B,X')$ is index consistent and has profile $\pi$.
\end{lemma}
The proof of Lemma~\ref{lem:goodtriple} can be found in Section~\ref{sec:lem6and10}.
Lastly, we put another restriction on $Z$ as to make sure that any relatively dense subset of any core creates many 
$k$-AP's for the second round. 

\begin{lemma}
\label{lem:corefocus}
For every $c_0>0$ and $\gamma>0$ there is a $\delta>0$ such that for $p\geq c_0 n^{-1/(k-1)}$ a.a.s.\  
the following holds. The size of $\Znp$ is at most $2pn$ and for every subset $S\subset \Znp$ of size
$|S|>\gamma pn$ the set
\[F(S)=\{z\in \Z/n\Z\colon \text{there are $a_1,\dots,a_{k-1}\in S$ which form a $k$-AP with } z\} \]
has size at least $\delta n$.
\end{lemma}
The proof of Lemma~\ref{lem:corefocus} can be found in Section~\ref{sec:corefocus}. We are now in the position to prove the main theorem.
\subsection{Proof of the main theorem}
\label{sec:proofmain}
The proof of the main theorem uses the lemmas introduced in the previous section and follows the scheme described.
\begin{proof}[Proof of Theorem~\ref{thm:main}]
For a given $k\geq 3$ assume for contradiction that $\Znp\rightarrow(k\text{-AP})$
does not have a sharp threshold. By
 Corollary~\ref{thm:BF} there exist constants $c_0$, $c_1$, $\alpha$, $\eps$, $\mu>0$ 
 and~$K$ and a function $p(n)=c(n)n^{-1/(k-1)}$ for some functin $c(n)$ satisfying
 $c_0\leq c(n)\leq c_1$  such that 
 for infinitely many $n$ there exists a subset $B\subset \Z/n\Z$ of size at most $K$
 with $B\not\rightarrow (k\text{-AP})_2$.

 We apply Lemma~\ref{lem:goodtriple} with $c_1$, $k$, $K$, and $\alpha$ to obtain $L$ and $\alpha'>0$. 
 For each $2\leq \ell\leq L$ we apply Lemma~\ref{lem:cores} with the constants $k$, $\ell$, $C_0=\alpha'/(2c_1)$, and
 $C_1=2(10k^3K)^{k-2}c_1^{(k-2)^2-1}$ to obtain $C'(\ell)$ and $\beta(\ell)>0$. Let
 \[
 	C'=\max\{C'(\ell)\colon 2\leq \ell\leq L\}
	\qand
	\beta=\min\{\beta(\ell)\colon 2\leq \ell\leq L\}
\] 
and apply Lemma~\ref{lem:corefocus} with $c_0$ and $\gamma=\beta/4$ to obtain
 $\delta>0$.
 
For each $n$ we define $\cZ_n$ to be the sets of subsets $Z\subset \Z/n\Z$ 
which satisfy the conclusions of Lemma~\ref{lem:degreeH}, Lemma~\ref{lem:goodtriple} and Lemma~\ref{lem:corefocus}
(with $\Znp$ replaced by $Z$) with the constants given and chosen from above.
As these lemmas assert properties of~$\Znp$ that hold a.a.s.\ 
we know that for sufficiently large $n$ we have $\Pr(\Znp\in\cZ_n) > 1-\mu$. Hence, 
by Corollary~\ref{thm:BF} there is an interacting triple $(Z,B,X)$ such that 
$|B|\leq K$, $|X|\geq \alpha n$ and $Z \in \cZ_n$. In particular, since $Z$ satisfies the conclusion of 
Lemma~\ref{lem:degreeH} and Lemma~\ref{lem:goodtriple}  there exists a set $Y_n$ of size at most 
$2n^{1-1/(k-1)}\log n$ and a profile $\pi$ of length $1\leq \ell\leq L$ and
a set $X'\subset X\setminus Y_n$ such that 
\begin{itemize}
\item $|X'|\geq \alpha' n$ 
\item the (interacting) triple $(Z,B,X')$ is index consistent and has profile $\pi$, 
\item the hypergraph $H=H(Z,B,X')$ satisfies  $2pn\geq v(H)=|Z|\geq pn/2$, the maximum degree of $H$ satisfies $\Delta_1(H)\leq 10k^3Kp^{k-2}n$
and $\Delta_2(H)\leq 8\log n$.
\end{itemize}
As $(Z,B,X')$ is regular we have $\ell\geq k-1\geq 2$ by definition.
Further, $H$ is an $\ell$-uniform hypergraph on $m=|Z|$ vertices 
which satisfies the assumptions of Lemma~\ref{lem:cores} with the constants chosen above.
Hence, by the conclusions of Lemma~\ref{lem:cores}
we obtain a family $\cC$ of cores, such that
\begin{enumerate}[label=\rmlabel]
\item for $t=1-1/(k-2)(\ell-1)$ we have \[|\cC|=\sum_{i=1}^{C'm^t}\binom{m}{i}\]
\item $|C|\ge \beta m$ for every $C\in \cC$, and
\item every hitting set of $H$ contains some $C$ from $\cC$.
\end{enumerate}
Let $\Phi$ be the set of all $k$-AP free colourings of $Z$. By Observation~\ref{obs:hittingset} and 
Observation~\ref{obs:hittingset2} 
we can associate to each $\phi\in\Phi$ a hitting set $A_{\phi}$ of $H$ such that any two colourings 
$\phi,\phi'\in\Phi$ agree on~$A_{\phi}\cap A_{\phi'}$, i.e.\ $\phi(z)=\phi'(z)$ for all $z\in A_{\phi}\cap A_{\phi'}$.
For any $C\in\cC$ we define $\Phi_C$ to be the set of~$\phi\in\Phi$ such that $C\subset A_{\phi}$ 
and obtain $\Phi=\bigcup_{C\in\cC}\Phi_C$.
Clearly, for any $C\in\cC$, any two~$\phi,\phi'\in\Phi_C$ agree on $C\subset A_{\phi}\cap A_{\phi'}$. Let
$B_C\subset C$ be the larger monochromatic subset of $C$ under (any)~$\phi\in \Phi_C$, say of colour blue. Then $B_C$ has size $|B_C|\geq |C|/2\geq \gamma pn$ and as
$Z\in\cZ_n$ we know by Lemma~\ref{lem:corefocus} that $|F(B_C)\setminus Z|>\delta n/2$. 
Let $\cP(C)$ denote the set of all $k$-APs contained in~$F(B_C)$.
By the quantitative version of Szemer\'edi's theorem (see \cite{Var59})
we know that there is an $\eta>0$ such that for sufficiently large $n$ we have $|\cP(C)|\geq \eta n^2$. 
Consider the second round exposure $U_C=\Znep\cap F(B_C)$ and let $t_i$ be the indicator
random variable for the event $i\in U_C$. We are interested in the probability that there is a 
$\phi\in\Phi_C$ which can be extended
to a $k$-AP free colouring of $Z\cup U_C$. To extend any colouring $\phi\in \Phi_C$ of $Z$ to a $k$-AP 
free colouring of~$Z\cup U_C$, however, it is necessary that $U_C\subset F(B_C)$ is  
completely coloured red, i.e.\ that $U_C$ does not contain any~$k$-AP. This probability can be bounded using Janson's 
inequality for  $X=\sum_{P\in \cP(C)}\prod_{i\in P}t_i$ given by
\[\pr{X=0}\leq \exp\left\{-\frac{\ex{X}^2}{2\Delta}\right\}\]
where 
\[\Delta=\sum_{A,B\in\cP(C)\colon A\cap B\neq\emptyset}\ex{\prod_{i\in A\cup B}t_i}\leq p^{2k-1}n^3+p^{k+1}k^2n^2\leq 2p^{2k-1}n^3\]
for large enough $n$. We obtain
\begin{multline*}\pr{\exists \phi\in\Phi_C\colon \phi \text{ can be extended to a $k$-AP free colouring of } Z\cup U_C}\\
\leq\pr{U_C \text{ does not contain a $k$-AP}}<\exp\{-\eta^2pn/4\}.\end{multline*}
Taking the union bound we conclude
\[\pr{(Z\cup \Znep)\not\rightarrow (k\text{-AP})_2}\leq |\cC|\exp\{-\eta^2pn/4\}\]
which goes to zero as $n$ goes to infinity. This, however, contradicts
property~\ref{item:BF2} of Corollary~\ref{thm:BF}.
\end{proof}

\section{Proofs of the Lemmas~\ref{lem:degreeH} and \ref{lem:goodtriple}}
\label{sec:lem6and10}
In this section we prove the lemmas introduced in the previous section.
We  start with some technical observations.
Given $B\subset\Z/n\Z$  and an element $z\in (\Z/n\Z)\setminus B$ let 
\[\cP(z,B)=\big\{P \subset \Z/n\Z\colon \text{ There is a } b\in B \text{ such that } P\cup\{z,b\} 
\text{ forms a } k\text{-AP}\big\}\]
and let $\cP(z,z',B)=\cP(z,B)\times\cP(z',B)$ where $z$ and $z'$ need not be distinct.
Further, let $\cP(z,B,\Z/n\Z)=\bigcup_{x\in \Z/n\Z}\cP(z,B+x)$ and in the same manner define
$\cP(z,z',B,\Z/n\Z)$.
\begin{fact}
 \label{fact:P0}
Let $z, z'\in \Z/n\Z$ and $a\in\Z/n\Z$ be given. Then  
\begin{enumerate}[label=\nlabel]
 \item the number of $P\in \cP(z,B,\Z/n\Z)$
such that $a\in P$ is at most $k^3|B|$.
 \item the number of pairs $(P,P')\in \cP(z,z',B,\Z/n\Z)$
such that $a\in P\cup P'$ is at most $2k^5|B|^2$.
\end{enumerate}
\end{fact}
\begin{proof}
 We only prove the second property. For that we count the number of pairs $(P,P')\in \cP(z,z',B,\Z/n\Z)$ 
 such that $a$ is in, say, $P$. Recall that there must exist $x\in\Z/n\Z$ and $b, b'\in B+x$ such that 
 $P\cup\{z, b\}$ and $P'\cup\{z',b'\}$ are both $k$-APs. 
 Choosing the positions of $z$ and $a$ uniquely determines the first $k$-AP. There are at
 most $(k-2)$ choices for $b$ to be contained in the $k$-AP and at most $|B|$ choices of 
 $x$ such that $b\in B+x$. 
 Each such choice determines $P$ and moreover, gives rise to
 at most $|B|$ choices for $b'$. Choosing the positions of $b'$ and $z'$ then determines the second $k$-AP, hence
 also $P'$.
 \end{proof}

We define
\[
 \cP_0(z,B)=\big\{P\in\cP(z,B)\colon P \text{ and } B \text{ are disjoint}\big\}\] and
\[ \cP_0(z,z',B)=\big\{(P,P')\in\cP(z,z',B)\colon P\cup\{z\}, P'\cup\{z'\} \text{ and } B \text{ are pairwise disjoint}\big\}.
\]

Further, let \[\cP_1(z,B)=\cP(z,B)\setminus \cP_0(z,B)\quad\text{ and }\quad\cP_1(z,z',B)=\cP(z,z',B)\setminus \cP_0(z,z',B).\]
For given sets  $A, B\subset \Z/n\Z$ we call $x\in \Z/n\Z$  \emph{bad} (with respect to 
$A$ and $B$) if 
\begin{enumerate}[label=\nlabel]
 \item there are  $z\in A\setminus B+x$ and  $P\in \cP_1(z,B+x)$ such that $P\cup\{z\}\subset A\cup B+x$ or
  \item there are $z,z'\in A\setminus B+x$ and $(P,P')\in \cP_1(z,z',B+x)$ such that 
  $P\cup P'\cup\{z,z'\}\subset A\cup B+x$.
\end{enumerate}

\begin{fact}
\label{fact:P1}
 Let $p\leq c_1 n^{-1/(k-1)}$ and let $B$ be a set of constant size. Let $Y_n$ be the set of bad elements
  with repect to $\Znp$ and $B$. Then a.a.s.\ $|Y_n|<p n\log n$.
 \end{fact}
\begin{proof}
 We will show that the expected size of $Y_n$ is of order $pn$ so that the statement 
 follows from Markov's inequality.

For a fixed $x$ we first deal with the case that $x$ is  bad due to the first property, i.e.\ there
is a $k$-AP $P\cup\{z,b\}$ with $b\in B+x$, $P$ intersecting $B+x$ and $z\in\Znp$ which does not belong to $B+x$.
Note that after choosing $b$, one common element of $P$ and $B+x$ and their positions the $k$-AP is uniquely 
determined. Then there are at most $(k-2)$ choices for $z$ each of which uniquely determines one $P$. 
Hence the probability that $x$ is bad due to the first property is at most $|B|k^3p$ and summing over all $x$
we conclude that the expected number of bad elements due to the first property is at most $|B|k^3pn$.

If $x$ is bad due to the second property
then there are two $k$-APs $P\cup\{z, b\}$ and $P'\cup\{z', b'\}$ such that
 two of the three sets
 $P \cup\{z\},P'\cup\{z'\}, B+x$ intersect  and $P\cup P'\cup\{z,z'\}\subset \Znp\cup B+x$
 where $z$ and $z'$ are not in $B+x$.
 
We distinguish two cases and first consider all tuples $(P,P',z,z',b,b')$ with the above mentioned 
properties such that $P$ (or $P'$ respectively)  does not intersect $B+x$. 
Note that with this additional property the probability that $x$ is bad due to $(P,P',z,z',b,b')$ is at most 
$p^{k}$ since $z,z'$ and $P$ all need to be in $\Znp$.
First, we count the number of such tuples with the additional property that $P'$ (or $P$ respectively) also has empty 
intersection with $B+x$. This implies that
$P$ and $P'$ must intersect and
in this case, choosing $b,b'\in B+x$ and one common element $a\in P\cap P'$ and the positions of $b,b',a$ in the 
$k$-AP's uniquely determines both $k$-APs.
After these choices there are at most $k^2$ choices for $z,z'$.
Hence, there are at most $|B|^2k^5n$ such tuples for a fixed $x$. 

Next,  we count the number of tuples $(P,P',z,z',b,b')$ with the property that $P'$ and $B+x$ intersect. 
In this case choosing $b'\in B+x$ and one element 
in $P'\cap B+x$ and their positions in the $k$-AP uniquely determines the second $k$-AP. 
Choosing $b\in B+x$, another element $a\in\Z/n\Z$, and their positions
determines the first $k$-AP. After these choices there are at most $k^2$ choices for $z,z'$ hence in total
there are at most $|B|^3k^6n$ such tuples for a fixed $x$.
We conclude that the expected number of bad $x$ due to tuples $(P,P',z,z',b,b')$ such that $P$ (or $P'$ respectively) 
does not intersect $B+x$ is at most $p^k2n(|B|^2k^5n+|B|^3k^6n)$.

It is left to consider the tuples $(P,P',z,z',b,b')$ such that $(P,P')\in \cP_1(z,z',B+x)$ and~$P$ and~$P'$ both intersect $B+x$. In this case
choosing $b, b'$, the element(s) in $P\cap B+x$ and $P'\cap B+x$ and their positions uniquely
determine the two $k$-APs.
Since $z,z'\in \Znp$ with probability $p^2$ the expected 
number of bad $x$ due to tuples $(P,P', z,z',b,b')$ with the above mentioned property is at most $|B|^4k^6p^2n$. 

Hence the expected size of $Y_n$ is 
\[
	|Y_n|<|B|k^3pn+2p^kn(|B|^2k^5n+|B|^3k^6n)+|B|^4k^6p^2n<6|B|^3k^6p^kn^2\,,
\] 
as claimed.
\end{proof}

\begin{fact}
\label{fact:matchings}
 Let $\ell\geq 1$ be an integer and let $F$ be an $\ell$-uniform hypergraph on the vertex set $\Z/n\Z$ which has maximum vertex degree at most
 $D$. Let $U=\Znp$ with $p=cn^{-1/(k-1)}$. Then
 with probability at most $2Dn^{-4}$ we have $e(F[U])> 5D(p^\ell n/\ell+\log n)$.
\end{fact}
\begin{proof}
 For $\ell=1$ the bound directly follows from Chernoff's bound
 \begin{align}\label{eq:chernoff}\pr{|X-\ex{X}|> t}\leq 2\exp\left\{-\frac{t^2}{2(\ex{X}+t/3)}\right\}\end{align}
 for a binomial distributed random variable $X$. For $\ell>1$
 we split the edges of $F$ into $i_0\leq D$ matchings $M_1,\dots,M_{i_0}$, 
 each of size at most $n/\ell$.
 For an edge $e\in E(F)$ let $t_e$ denote the random variable indicating that $e\in E(F[U])$.
 Then  $\pr{t_e=1}=p^\ell$ and we set $s=4\max\{p^\ell n/\ell,\log n\}$. By Chernoff's bound
 we have
\[\pr{\exists i\leq i_0\colon \sum_{e\in M_i}t_e>s}\leq 2Dn^{-4}.\]
This finishes the proof since $e(F[U])=\sum_{i\in[i_0]}\sum_{e\in M_i}t_e$ which exceeds $Ds$	 with 
probability at most $2Dn^{-4}.$
\end{proof}
\subsection*{Proof of Lemma~\ref{lem:degreeH} and Lemma~\ref{lem:goodtriple}}
Based on the preparation from above  we give the proof of  Lemma~\ref{lem:degreeH} and Lemma~\ref{lem:goodtriple} in this 
section.
\begin{proof}[Proof of Lemma~\ref{lem:degreeH}]
Let $t_i$ be the indicator random variable for the event $i\in \Znp$. 
Consequently, $v(H)=\sum_{i\in\Z/n\Z}t_i$ is binomially distributed and the 
first property directly follows from Chernoff's bound \eqref{eq:chernoff}.

For the second and third properties we first consider all elements from $\Z/n\Z$ which are bad with respect
to $\Znp$ and $B$. By Fact~\ref{fact:P1} we know that a.a.s.\ the set $Y_n$ of bad elements has size at most 
$n^{1-1/(k-1)}\log n$ and
in the following we will condition on this event. 

We consider the degree and co-degree in the hypergraph $H(\Znp,B,(\Z/n\Z)\setminus Y_n)$.
Let $z, z'\in\Z/n\Z$ be given.  
If $z$ is contained in an edge 
$M_x=M(\Znp,B+x)$ then there is an element $P\in\cP(z,B+x)$ 
such that $P\subset \Znp\cup B+x$. It is sufficient to focus on those $P\in\cP_0(z,B+x)$ since after removing $Y_n$  
those $P\in\cP_1(z,B+x)$ will have no contribution. Hence we can assume
$P\subset \Znp$ or equivalently $\prod_{i\in P}t_i=1$.
Further, there are at most three different values of $y$ such that $P$ is contained 
in $\cP_0(z,B+y)$.
Hence, letting $\cP_0(z,B,\Z/n\Z)=\bigcup_{x\in\Z/n\Z}\cP_0(z,B+x)$ we can bound the degree of $z$ by 
\begin{align}\label{eq:deg}\deg(z)\leq 3\sum_{P\in\cP_0(z,B,\Z/n\Z)}\prod_{i\in P}t_i.\end{align}
 Similarly, if $z,z'$ are contained in an edge $M_x$ then there is a pair $(P,P')\in \cP_0(z,z',B+x)$ such that
 $P\cup P'\subset \Znp\cup B+x$, i.e.\ $\prod_{i\in P\cup P'}t_i=1$. We obtain
\begin{align}\label{eq:codeg}\codeg(z_1,z_2)\leq 9\sum_{(P,P')\in \cP_0(z,z',B,\Z/n\Z)}\prod_{i\in P\cup P'}t_i\end{align}

We consider $\cP_0(z,B,\Z/n\Z)$ (respectively, $\cP_0(z,z',B,\Z/n\Z)$) as a $(k-2)$-uniform (resp., $(2k-4)$-uniform) 
hypergraph on the vertex set $\Z/n\Z$. By Fact~\ref{fact:P0} we know that the maximum degree of the 
hypergraph is at most $k^3K$ (resp.\ $5k^5K^2$). By Fact~\ref{fact:matchings} the probability that 
$\deg(z)>10k^3Kp^{k-2}n$ or $\codeg(z,z')>8\log n$ is at most $2k^3Kn^{-4}$. Taking the union bound over all 
elements and all pairs of $\Z/n\Z$ we obtain the desired property.
\end{proof}

\begin{proof}[Proof of Lemma~\ref{lem:goodtriple}]
 For given $k, c, K$ and $\alpha$ set $L=20c^{k-1}k^2/\alpha$ and $\alpha'=\alpha/2(LK)^L$
 and let some linear order on $\Zn$ be given.
 First we choose $Y_n$ so as to guarantee regularity of the triple $(\Znp,B,(\Z/n\Z)\setminus Y_n)$. Note that there 
 are two sources of irregularity: $k$-APs containing
 two elements from $B+x$ for some $x\in \Z/n\Z$ and pairs of $k$-APs with one common element in $\Znp$ and 
 each containing one element in $B+x$ for some $x\in \Z/n\Z$. These are ruled out by removing all $x$ which are 
 bad with respect to $\Znp$ and $B$. By Fact~\ref{fact:P1} a.a.s.\ the set $Y_n$ of bad elements has 
 size at most $n^{1-1/(k-1)}\log n$.

 Further, we count the number of $(k-1)$-element sets in $\Znp$ 
 which arise from $k$-APs with one element removed. The expected number of such sets is at most
 $p^{k-1}kn^2$ and the variance is of order at most $p^{k-2}n^2$. Hence, by Chebyshev's inequality 
 the number of such sets in $\Znp$ is at most $2c^{k-1}kn$ asymptotically almost surely. 
 
Consider any set $A\subset \Z/n\Z$ which posseses the two properties mentioned above: there is a set $Y_n$
of size at most $n^{1-1/(k-1)}\log n$ such that $(A,B,(\Z/n\Z)\setminus Y_n)$ is regular and the number of
$(k-1)$-element sets in $A$ which arise from $k$-APs with one element removed is at most~$2c^{k-1}kn$. 
Let $X\subset\Z/n\Z$ of size $|X|\geq \alpha n$ be given. 
For every $x\in X\setminus Y_n$ let $\ell_x$ denote 
the size of $M_x=M(A,B,X)$. Then there are at least $\ell_x/(k-1)$
sets of size $(k-1)$ contained in $M_x$ each forming a 
$k$-AP with an element in $B+x$. Further, each such set is contained in~$B+x'$ for at most three $x'\in \Z/n\Z$,
hence, $\sum_{x\in X\setminus Y_n}\ell_x/(k-1)\leq 6c^{k-1}kn$ and we conclude that the number of 
$x\in X\setminus Y_n$ such that  
$\ell_x>20c^{k-1}k^2/\alpha=L$ is at most $\alpha n/3$.
Moreover, there are at most $|B|^\ell$ distinct profiles of length $\ell$, hence there is a profile $\pi$ of 
length $\ell\leq L$ and a 
set $U\subset X\setminus Y_n$ of size $|U|\geq \alpha n/2K^{L}$ such that $(A,B,U)$ has profile $\pi$.

To obtain a set $X'\subset U$ such that $(A,B,X')$ is index consistent we consider a random partition of 
$A$ into $\ell$ classes $(V_1,\dots,V_{\ell})$.
We say that an edge $M_x\in H(A,B,U)$ with the elements $z_1<\dots<z_{\ell}$ survives if $z_i\in V_i$ for all 
$i\in[\ell]$. The probability of survival is  
$\ell^\ell$, hence there is a partition such that at least $|U|/\ell^{\ell}$ edges survives. Choosing the corresponding 
set $X'\subset U$ yields a set with the desired properties.  
\end{proof}

\section{Proof of Lemma~\ref{lem:cores}}\label{sec:cores}
In this section we prove Lemma~\ref{lem:cores}. The proof relies crucially on 
a structural theorem of Balogh, Morris and Samotij~\cite{BMS12} which we state in the following.
Let $H$ be a uniform hypergraph with vertex set $V$ and let $\cF$
be an increasing family of subsets of $V$ and $\eps\in(0,1]$. The 
hypergraph $H$ is called $(\cF,\eps)$-dense if for every $A\in \cF$ 
\[
e(H[A])\ge \eps e(H).
\]
Further, let $\cI(H)$ denote the set of all independent sets of $H$.
The so-called \emph{container theorem} by Balogh, Morris and Samotij~\cite{BMS12}*{Theorem~2.2} reads as follows.
\begin{theorem}[Container theorem]\label{thm:BMS}
For every $\ell\in\Nat$ and all positive $c$ and $\eps$, there exists a positive 
constant $c'$ such that the following holds. Let $H$ be an $\ell$-uniform hypergraph and let~$\cF$ be an increasing family of subsets of $V$ such that $|A|\ge \eps v(H)$
 for all $A\in\cF$. Suppose that $H$ is $(\cF,\eps)$-dense and $p\in(0,1)$ is such that 
 for every $t\in[\ell]$ the maximum $t$-degree $\Delta_{t}(H)$ of $H$ satisfies
 \[
 \Delta_{t}(H)
 =
 \max_{\subalign{T&\subseteq V\\ |T|&=t}}\big\{|\{e\supseteq T\colon e\in E(H)\}|\big\}
 \le c p^{t-1}\frac{e(H)}{v(H)}\,. \]
 Then there is a family $\cS\subseteq \binom{V(H)}{\le c'pv(H)}$ and 
 functions $f\colon \cS\to\overline{\cF}$ and $g\colon\cI(H)\to \cS$ such that for every $I\in \cI(H)$,
 \[
 g(I)\subseteq I \qand I\setminus g(I)\subseteq f(g(I)).
 \]
\end{theorem}

Theorem~\ref{thm:BMS} roughly says that if an uniform hypergraph $H$ satisfies certain 
conditions then the set of independent sets $\cI(H)$ of $H$  
can be ``captured'' by a family $\cS$ consisting of small sets. Indeed, 
every independent set $I\in \cI(H)$ contains 
a (small) set $g(I)\in\cS$ and the remaining elements of $I$ must come from a set determined by $g(I)$.

We are now in a position to derive Lemma~\ref{lem:cores} from Theorem~\ref{thm:BMS}.

\begin{proof}[Proof of Lemma~\ref{lem:cores}]
Given the constants $k, \ell, C_0$ and $C_1$ we apply Theorem~\ref{thm:BMS} with $\eps=1/2$, and 
$c=\max\{1,C_1/C_0\}$ to obtain $c'$. We let $C'=c'$ and $\beta=\min\{1/4,C_0/(4C_1)\}>0$.

We define the increasing family $\cF$ by
\[
\cF=\{A\colon A\subset V(H), \text{ s.t. }|A|\ge m/2 \text{ and } e(H[A])\ge e(H)/2\}.
\]
Clearly, $H$ is $(\cF,1/2)$-dense and, moreover,  any set $A\in\ocF$  has size less 
than $m/2$ or we have $e(H[A])<e(H)/2$. Thus, at least $e(H)/2$ edges 
of $H$ are incident to the vertices outside 
of~$A$. Therefore, $\Delta_1(H)(m-|A|)\ge e(H)/2\geq C_0 m^{1+1/(k-2)}/2$ and with 
 $\Delta_1(H)\le C_1 m^{1/(k-2)}$ we conclude: 
 \[
    |A|\le (1-2\beta)m.
 \]

We define $p=m^{-1/(k-2)(\ell-1)}$ so that 
$\Delta_t(H)\leq cp^{t-1}\frac{e(H)}{v(H)}$ for all $t\in[\ell]$. 
In fact, for $t=1$ this follows directly from the bound $\Delta_1(H)\leq C_1m^\frac{1}{k-2}$ given by  
the assumption of Lemma~\ref{lem:cores}. For $t=2,\dots,\l-1$ we use $\Delta_t(H)\leq \Delta_2(H)$
and on $\Delta_2(H)$ given by the assumption of Lemma~\ref{lem:cores}. Finally, for $t=\l$ we note that 
$\Delta_\l(H)=1$ and the desired bound follows again from the choices of~$p$ and~$c$.

Thus there exist a family $\cS$ 
 and functions  $f\colon \cS\to \ocF$ and $g\colon \cI(H)\to\cS$ 
 with the properties described in Theorem~\ref{thm:BMS}.
We define
\[
   \cA=\{{S\cup f(S)}\colon S\in\im(g)\},
\]
where $\im(g)$ is the image of $g$.
Our cores will be the complements of the elements of $\cA$, 
\[
\cC = \{V(H) \setminus {A} \colon A \in \cA\}.
\]
Since $|\cC|=|\cA|\leq|\cS|$, we infer \ref{it:cores1} of Lemma~\ref{lem:cores}. Further, 
every $A\in\cA$ has size at most 
\[
	(1-2\beta)m+C'pm\le (1-\beta)m
\] 
which yields the  property \ref{it:cores2} of Lemma~\ref{lem:cores}.

Finally, by the properties of the functions $f$ and $g$, every independent set
$I$ is contained in~$A=g(I)\cup f(g(I))$, so, by taking complements, every hitting set contains an element of $\cC$
which completes the property \ref{it:cores3} of Lemma~\ref{lem:cores}.
\end{proof}

\section{Proof of Lemma~\ref{lem:corefocus}}
\label{sec:corefocus}
In this section we prove Lemma~\ref{lem:corefocus} which relies on the following result by the last author~\cite{mathias}
(see also~\cites{CG11,BMS12,ST12}).
\begin{theorem}
\label{thm:mathias}
 For every integer $k\geq 3$ and every $\gamma\in(0,1)$ there exists  $C$ and  $\xi>0$ such that for every sequence 
 $p=p_n\geq Cn^{-1/(k-1)}$ the following holds a.a.s.\:
 Every subset of $\Znp$ of size at least $\gamma pn$ contains at least $\xi p^kn^2$ arithmetic progression
 of length $k$.
\end{theorem}

With this result at hand we now prove Lemma~\ref{lem:corefocus}.
\begin{proof}[Proof of Lemma~\ref{lem:corefocus}]
The upper bound on the size of $\Znp$ follows from Chernoff's bound \eqref{eq:chernoff}.
For the second property let
 $k\geq 3$ and $\gamma$ be given. For $k\geq 4$ we apply Theorem~\ref{thm:mathias} with $k-1$ and $\gamma$ 
to obtain $C$ and $\xi$. We may assume that $\xi\leq \gamma^2/4$ and we note that in the case $k\geq4$ 
we have $p\geq c_0n^{-1/(k-1)}>Cn^{-1/(k-2)}$ for sufficiently large n.
We choose $\delta=\xi^2/20$.

Let $S\subset \Znp$ be a set of size at least $\gamma pn$. 
For a given $i\in\Z/n\Z$ let $\deg(i)$ denote the number of $(k-1)$-APs in $S$ which form a $k$-AP with $i$. Note
that $i\in F(S)$ if $\deg(i)\neq 0$.
Then a.a.s\ we have 
\[\sum_{i\in\Z/n\Z}\deg(i)\geq \xi p^{k-1}n^2\]
which holds trivially 
for the case $k=3$ due to  $\xi\leq \gamma^2/4$ and which is a consequence of 
Theorem~\ref{thm:mathias} for larger $k$.

Further, let $W=\sum_{i\in\Z/n\Z}\binom{\deg(i)}2$. Then, as $S\subset \Znp$
\[\ex{W}\leq n(p^{2(k-1)}n^2+ k^2p^{k-1-\lfloor\frac{k-1}2\rfloor} n)\leq 2p^{2(k-1)}n^3\] and 
its variance is of order  at most $p^{2(k-1)}n^3$.
Hence, by Chebyshev's inequality we have $\sum_{i\in\Z/n\Z}\binom{\deg(i)}2\leq 4p^{2(k-1)}n^3$ asymptotically almost surely.

Altogether we obtain
\[\binom{\xi p^{k-1}n^2}2\leq \binom{\sum_{i\in F(S)}\deg(i)}2\leq |F(S)|\sum_{i\in F(S)}\binom{\deg(i)}2\leq |F(S)|4p^{2(k-1)}n^3\]
and we conclude that $|F(S)|\geq \xi^2n/20=\delta n$.
\end{proof}
\subsection*{Acknowledgement}
We thank the anonymous referee for her or his constructive and helpful remarks.

\begin{dajauthors}
\begin{authorinfo}[EF]
	Ehud Friedgut\\
	Faculty of Mathematics and Computer Science\\
	Weizmann Institute of Science\\
	Rehovot, Israel\\
	ehud\imagedot{}friedgut\imageat{}weizmann\imagedot{}ac\imagedot{}il
\end{authorinfo}

\begin{authorinfo}[Hiep]
	Hi\d{\^e}p H\`an\\
	Instituto de Matem\'aticas\\ 
	Pontificia Universidad Cat\'olica de Valpara\'\i{}so\\ 
	Valpara\'\i{}so, Chile\\
	han\imagedot{}hiep\imageat{}googlemail\imagedot{}com
\end{authorinfo}

\begin{authorinfo}[YP]
	Yury Person\\
	Institut f\"ur Mathematik\\ 
	Goethe-Universit\"at\\
	Frankfurt am Main, Germany\\
	person\imageat{}math\imagedot{}uni-frankfurt\imagedot{}de\\
	\url{http://www.math.uni-frankfurt.de/~person/}
\end{authorinfo}

\begin{authorinfo}[MS]
	Mathias Schacht\\
	Fachbereich Mathematik\\
	Universit\"at Hamburg\\
	Hamburg, Germany\\
	schacht\imageat{}math\imagedot{}uni-hamburg\imagedot{}de\\
	\url{http://www.math.uni-hamburg.de/home/schacht/}
\end{authorinfo}
\end{dajauthors}

\end{document}